\documentclass[graybox]{svmult}
\usepackage{amssymb}%SB
\usepackage{amsfonts}%SB
\usepackage{amsmath}%SB
\usepackage{tikz}%SB 
\usepackage{mathptmx}       % selects Times Roman as basic font
\usepackage{helvet}         % selects Helvetica as sans-serif font
\usepackage{courier}        % selects Courier as typewriter font
\usepackage{makeidx}         % allows index generation
\usepackage{graphicx}        % standard LaTeX graphics tool
\usepackage{multicol}        % used for the two-column index
\usepackage[bottom]{footmisc}% places footnotes at page bottom
\makeindex             % used for the subject index
\def\Ucal{\mathcal{U}}

\def\Xcal{\mathcal{X}}
\def\Zcal{\mathcal{Z}}

\renewcommand{\to}{\rightarrow}

\newcommand{\spn}{\operatorname{Span}}

\newcommand{\grad}{\boldsymbol{\nabla}}

\newcommand{\FF}{\mathbb{F}}

\newcommand{\R}{\mathbb{R}}

\begin{document}
\title*{Johnson-Segalman -- Saint-Venant equations
for viscoelastic shallow %gravity
flows in the %purely 
elastic limit}
\titlerunning{Riemann problems for Johnson-Segalman -- Saint-Venant fluids}
\author{S\'ebastien Boyaval}
%\authorrunning{S\'ebastien Boyaval}
\institute{S\'ebastien Boyaval \at 
Laboratoire d'hydraulique Saint-Venant (Ecole des Ponts ParisTech -- EDF R\& D -- CEREMA) Universit\'e Paris-Est, EDF'lab 6 quai Watier 78401 Chatou Cedex France, 
\& INRIA Paris MATHERIALS 
\email{sebastien.boyaval@enpc.fr}}
\maketitle
%\abstract*{Please use the 'starred' version of the new Springer \texttt{abstract} command for typesetting the text of the online abstracts (cf. source file of this chapter template \texttt{abstract}) and include them with the source files of your manuscript. Use the plain \texttt{abstract} command if the abstract is also to appear in the printed version of the book.}
\abstract{The shallow-water equations of Saint-Venant, %which have often been
often used to model the long-wave dynamics of free-surface %Newtonian 
flows driven %mainly 
by inertia and hydrostatic pressure, can be generalized to account for the elongational rheology of non-Newtonian fluids too.
We consider here the $4\times 4$ shallow-water equations generalized to \emph{viscoelastic} fluids % under gravity on a flat bottom
using the Johnson-Segalman model % of non-Newtonian fluids, 
in the elastic limit (i.e. at infinitely-large Deborah number, when source terms vanish).
The % (homogeneous) 
system of nonlinear first-order equations is hyperbolic when the \emph{slip parameter} is small $\zeta\le\frac12$ 
($\zeta=1$ is the corotational case and $\zeta=0$ the upper-convected Maxwell case). % Oldroyd-B with zero retardation time
Moreover, it is naturally endowed with a mathematical entropy %stemming from 
(a physical free-energy). % somehow {\it non-conservative}.
When $\zeta\le\frac12$ and
for any initial data excluding vacuum,
we construct here, when elasticity $G>0$ is non-zero,
the unique solution to the Riemann problem under Lax admissibility conditions.
The standard Saint-Venant case is recovered when $G\to0$ for small %initial 
data.}
% which has solutions only for small initial data and accomodates vacuum, then appears as a / some kind of singular limit
\section{Setting of the problem} %%%%%%%%%%%%%%%%%%%%%%%%%%%%%%%%%%%%%%%%%%%%%%%%%%%%%%%%%%%%%%%
\sectionmark{Setting}
\label{sec:pb}
The well-known one-dimensional %(1D) 
shallow-water equations % SWE
of Saint-Venant
\begin{eqnarray}
\label{eq:saintvenant1}
% \left\lbrace
% \begin{aligned}
% & 
\partial_{t}h + \partial_{x}(h u) = 0
\\
% & 
\partial_{t}(h u) + \partial_{x}\left( h u^2 + gh^2/2 \right) = 0
% % & \partial_{t}h + u\partial_{x}h + h\partial_{x}u = 0,
% % & \partial_{t}u + u\partial_{x}u + g\partial_{x}h = 0,
% \end{aligned}
% \right.
\label{eq:saintvenant2}
\end{eqnarray}
%are used to 
model the dynamics of the mean depth $h(t,x)>0$ 
of a perfect % inviscid % Newtonian ??
fluid flowing with mean velocity $u(t,x)$
on %in 
a flat open channel with uniform cross section along a straight axis $\vec{e}_x$,
under gravity (perpendicular to $\vec{e}_x$, with constant magnitude $g$). % , see e.g.

Now, following the interpretation of (\ref{eq:saintvenant1}--\ref{eq:saintvenant2}) as an approximation of 
the depth-averaged Free-Surface Navier-Stokes (FSNS)
equations governing Newtonian fluids, % with friction at bottom
and starting depth-averageing from the % latter Free-Surface Navier-Stokes 
FSNS/Upper-Convected-Maxwell(UCM) system of %coupled 
equations for %governing {\it non-Newtonian}
(linear) \emph{viscoelastic} fluids \cite{bouchut-boyaval-2013,bouchut-boyaval-2015},
% it is quite straightforward to
one can in fact derive 
% a similarly reduced formulation 
% more general
a \emph{generalized Saint-Venant} (gSV) system of shallow-water-type equations
\begin{eqnarray}
\label{eq:gsv1}
\partial_{t}h + \partial_{x}(h u) & = & 0
\\
\partial_{t}(h u) + \partial_{x}\left( h u^2 + gh^2/2 + h N \right) & = & 0
\label{eq:gsv2}
\end{eqnarray}
where the %non-zero 
normal stress difference term in the momentum balance %(accounting for non-Netwonian {\it elongational} effects) 
$N= \tau_{zz}-\tau_{xx}$ 
is function of additional extra-stress variables $\tau_{zz}(t,x),\tau_{xx}(t,x)$ governed by e.g.
\begin{eqnarray}
\label{eq:gsv3}
\tau_{xx} + \lambda ( \partial_{t}\tau_{xx} + u\partial_{x}\tau_{xx}-2\tau_{xx}\partial_{x}u ) & = & \nu\partial_{x}u
\\
\tau_{zz} + \lambda ( \partial_{t}\tau_{zz} + u\partial_{x}\tau_{zz}+2\tau_{zz}\partial_{x}u ) & = & -\nu\partial_{x}u
\label{eq:gsv4}
\end{eqnarray}
i.e. depth-averaged UCM equations modelling % account for 
\emph{elongational} viscoelastic effects.
%\emph{FUTURE: In 2D, compute the extensional flow of a slighlty compressible/FS fluid around a cylinder with small shear at low Mach number.}

When the relaxation time is small $\lambda\to0$ 
(i.e. the Deborah number, when $\lambda>0$ is non-dimensionalized with respect to a time scale characteristic of the flow)
the system (\ref{eq:gsv1}-\ref{eq:gsv2}-\ref{eq:gsv3}-\ref{eq:gsv4}) %formally 
converges to standard viscous Saint-Venant equations with %a kinematics
viscosity $\nu\ge0$.
When the relaxation time %is large $\lambda\to+\infty$ as well as the viscosity $\nu\to+\infty$, 
{\it and the viscosity} are equivalently large $\lambda\sim\nu\to+\infty$,
the system (\ref{eq:gsv1}--\ref{eq:gsv4}) 
converges to elastic Saint-Venant equations (in Eulerian formulation, see e.g. \cite{picasso-2016})
with elasticity %modulus
$G=\nu/(2\lambda)\ge0$,
which coincides with the homogeneous version of the system (\ref{eq:gsv7}--\ref{eq:gsv8}) (i.e. when the souce term vanish)
\begin{eqnarray} 
% \label{eq:gsv5}
% \partial_{t}h + \partial_{x}(h u) & = & 0
% \\
% \label{eq:gsv6}
% \partial_{t}(h u) + \partial_{x}\left( h u^2 + gh^2/2 + Gh(\sigma_{zz}-\sigma_{xx}) \right) & = & 0
% \\
% \partial_{t}\tau_{xx} + u\partial_{x}\tau_{xx}-2\tau_{xx}\partial_{x}u = ({\nu\partial_{x}u-\tau_{xx}})/\lambda
% \\
% \partial_{t}\tau_{zz} + u\partial_{x}\tau_{zz}+2\tau_{zz}\partial_{x}u = ({-\nu\partial_{x}u-\tau_{zz}})/\lambda
% 
\label{eq:gsv7}
\partial_{t}\sigma_{xx} + u\partial_{x}\sigma_{xx}-2\sigma_{xx}\partial_{x}u & = & ({1-\sigma_{xx}})/\lambda
\\
\label{eq:gsv8}
\partial_{t}\sigma_{zz} + u\partial_{x}\sigma_{zz}+2\sigma_{zz}\partial_{x}u & = & ({1-\sigma_{zz}})/\lambda
\end{eqnarray}
obtained after rewriting (\ref{eq:gsv3}--\ref{eq:gsv4})
using %$\tau=G(\sigma-1)$
%the conformation variable 
$N=G(\sigma_{zz}-\sigma_{xx})$, 
and
% $\sigma_{\cdot}=\frac1G\tau_{\cdot}+1$
% $\sigma_{xx,zz}=\frac1G\tau_{xx,zz}+1$.
$\tau_{xx,zz} = G(\sigma_{xx,zz}-1)$.

% also has a conservative formulation
% on noting (\ref{eq:gsv7}--\ref{eq:gsv8}) %formally
% rewrite
% (without source term in the elastic limit $\lamba\to\infty$ where the right-hand-side in (\ref{eq:gsv7}--\ref{eq:gsv8}) is zero)
% \begin{eqnarray} 
% \label{eq:gsv9bis}
% \partial_{t}\left(h\log(h^2\sigma_{xx})\right) + \partial_{x}\left(hu\log(h^2\sigma_{xx})\right) & = & h(\sigma_{xx}^{-1}-1)/\lambda
% \\
% \label{eq:gsv10bis}
% \partial_{t}\left(h\log(h^{-2}\sigma_{zz})\right) + \partial_{x}\left(hu\log(h^{-2}\sigma_{zz})\right) & = & h(\sigma_{zz}^{-1}-1)/\lambda
% \end{eqnarray}

More general evolution equations of differential rate-type for the extra-stress,
the Johnson-Segalman (JS) equations with slip parameter $\zeta\in[0,2]$,
can also be coupled to FSNS before depth-averageing. 
% the depth-averageing procedure.
% invoking the more general Gordon-Schowalter objective derivatives, which contains Oldroyd upper-convected, lower-convected and corotational ones, see
In fact (\ref{eq:gsv7}--\ref{eq:gsv8}) %coincide with
arise in the specific case $\zeta=0$ (upper-convected Gordon-Schowalter derivative)
for gSV %generalized Saint-Venant 
system (\ref{eq:gsv1}--\ref{eq:gsv2}--\ref{eq:gsv9}--\ref{eq:gsv10}) 
\begin{eqnarray}
\label{eq:gsv9}
\partial_{t}\sigma_{xx} + u\partial_{x}\sigma_{xx} +2(\zeta-1)\sigma_{xx}\partial_{x}u & = & ({1-\sigma_{xx}})/\lambda
\\
\label{eq:gsv10}
\partial_{t}\sigma_{zz} + u\partial_{x}\sigma_{zz} +2(1-\zeta)\sigma_{zz}\partial_{x}u & = & ({1-\sigma_{zz}})/\lambda
\end{eqnarray}
% The gSV % generalized Saint-Venant 
% system (\ref{eq:gsv1}--\ref{eq:gsv2}--\ref{eq:gsv9}--\ref{eq:gsv10}) 
% for non-Newtonian shallow % thin-layer free-surface
% flows
that accounts for %the 
{\it linear}
viscoelastic elongational effects %that have been 
standardly established for e.g. polymeric liquids \cite{bird-curtiss-armstrong-hassager-1987a}.
% although % other
The gSV system with JS 
is % seems to be
already
an interesting starting point % (i.e. physically-validated)
for mathematical studies,
although % still more complicated equations exist that contain more % (nonlinear) physics 
it could %of course
still be further complicated to account for more established physics ; we refer to \cite{bird-curtiss-armstrong-hassager-1987a} for details. % than JS

% \emph{FUTURE: gSV also coincides with generalized isentropic Euler equations modelling compressible gas dynamics with viscoelastic effects in the specific case where the pressure law is quadratic in mass density.}

In the following, we consider the Cauchy problem on $t\ge0$ for the quasilinear %(non-conservative) % first-order % 1D 
gSV system (\ref{eq:gsv1}--\ref{eq:gsv2}--\ref{eq:gsv9}--\ref{eq:gsv10}) % simply gSV in short
when it is supplied by an initial condition with bounded total variation.
% While it is standard that smooth classical solutions do not exist globally 
Weak solutions % in the standard distributional sense
with % support on $\R$, a priori only local in time and
bounded variations (BV) on $\R\ni x$
can be constructed % by various methods 
for %non-conservative
quasilinear systems % in 1D
% for instance using
% Glimm's Random-Choice % provided total variation is small initially -- requires the exact Riemann solution -- solution has bounded variation !! thus also in L^p, p\ge 1
% Front-tracking % also small TV initially, and solution BV, but requires only an approximate Riemann solver -- basis of Bressan and Bianchini proof -- 
% vanishing-viscosity, % also provided total variation is small initially -- in fact L^1 norm for contraction, see also Bianchini and Bressan where it is compared to Front-tracking
% compensated compactness % for weaker solutions in L^p, p>1, but this requires assumptions and has a smaller scope
provided the system is {\it strictly hyperbolic},
in particular when characteristic fields are %either 
genuinely nonlinear or linearly degenerate %, see e.g. 
\cite{dafermos-2000,lefloch-2002}.
% see \cite[T1, Chap. 3]{serre-1996} for a definition

First, we show that gSV is hyperbolic for all $h\ge0,\sigma_{xx}>0,\sigma_{zz}>0$ provided $\zeta\le\frac12$~;
\emph{strictly} provided $h>0$, or % when $h=0$
$G>0$ and, $\zeta>0$ or $\sigma_{xx}\neq\sigma_{zz}$. % when $\zeta=0$ (UCM)
Next, assuming $\zeta\le\frac12$ and $h>0$,
we construct univoque % BV 
gSV solutions %also need to 
% on %supplying it with rule
% requiring 
guided by the %entropy 
dissipation~rule
\begin{equation}
\label{secondprinciple}
 \partial_{t}F 
+ \partial_{x}\left(u(F + P % gh^2/2 + hN % = P
  )\right) \le %\frac{G}{2\lambda}
  G\left( 2-\sigma_{xx}-\sigma_{xx}^{-1} + 2-\sigma_{zz}-\sigma_{zz}^{-1} \right)/(2\lambda)
\end{equation}
for the same mathematical 
entropy $F$ %straightforward extension
%extended to $\zeta\in\left[0,\frac12\right]$ from
as for 
$\zeta=0$ \cite{bouchut-boyaval-2013}
% the straightforward extension of the % physically-relevant notion of 
% free-energy 
% already proposed in \cite{bouchut-boyaval-2013} for the case $\zeta=0$.
% which implies the preservation of the set $h\ge0,\sigma_{xx}>0,\sigma_{zz}>0$ in particular ??
as admissibility criterion
\begin{equation}
\label{freeenergy}
F = h\left(u^2 + gh +  G(\sigma_{xx}+\sigma_{zz}-\ln\sigma_{xx}-\ln\sigma_{zz}-2) \right)/2 \,,
\end{equation}
denoting $P=gh^2/2+hN$.
Smooth gSV solutions obviously satisfy the \emph{equality} \eqref{secondprinciple}.
When $h,\sigma_{xx},\sigma_{zz}>0$,
gSV %is actually
% with two genuinely nonlinear characteristic fields 
% and one linearly degenerate characteristic field of multiplicity two,
reads as a system of conservation laws
% preserved on $\RR\ni x$ by the dynamics for classical solutions at least ?
% an invariant region %\cite[4.3.3 p.169]{cherrier-milani-2012} ????????????
% so that
% while
rewriting (\ref{eq:gsv9}--\ref{eq:gsv10}) % in the conservative formulation
\begin{eqnarray} 
\label{eq:gsv9bis}
\partial_{t}\left(h\log(h^{2(1-\xi)}\sigma_{xx})\right) + \partial_{x}\left(hu\log(h^{2(1-\xi)}\sigma_{xx})\right) & = & h(\sigma_{xx}^{-1}-1)/\lambda ,
\\
\label{eq:gsv10bis}
\partial_{t}\left(h\log(h^{2(\xi-1)}\sigma_{zz})\right) + \partial_{x}\left(hu\log(h^{2(\xi-1)}\sigma_{zz})\right) & = & h(\sigma_{zz}^{-1}-1)/\lambda .
\end{eqnarray}
But 
% the usual standard construction methods of weak % BV 
% solutions for system of conservation laws % (invoking the solutions to Riemann problems as building blocks)
% are famously non-univoque
% and could not invoke our % physically-natural 
% entropy-dissipation rule in a second step (to weed out non-physical weak solutions),
% because our mathematical entropy (a free-energy in physics)
% is not convex with respect to the conservative variables, see \cite{bouchut-boyaval-2013},
% while it is convex with respect to $h,u,\sigma_{xx},\sigma_{zz}$,
whereas %our mathematical entropy 
$F$ % (a free-energy in physics)
is convex in %with respect to 
e.g. $(h,hu,h\sigma_{xx},h\sigma_{zz})$,
% or $U=(h,u,\sigma_{xx},\sigma_{zz})$ % $h,u,\sigma_{xx},\sigma_{zz}$ % CHECK !!!!!!!!!
see \cite{bouchut-boyaval-2013} when $\xi=0$,
it % is never 
cannot be convex with respect to any 
%conservative 
variable % like $(h,hu,h\log(h^{2(1-\xi)}\sigma_{xx}),h\log(h^{2(\xi-1)}\sigma_{zz}))$
$V = (h,hu,h\Xcal(\sigma_{xx}h^{2(1-\zeta)}),h \Zcal(\sigma_{zz}h^{2(\zeta-1)}))$ 
using smooth $\Xcal,\Zcal\in C^1(\R_\star^+,\R_\star^+)$ % for example: $\Xcal(v)=v$ and $\Zcal(v)=v^{-1}$, 
% there exist a %conservative ``flux function'' here
such that the system rewrites
\begin{equation}
\label{conservative}
\partial_t V + \partial_x \FF(V) = 0\,,
\end{equation}
% with vector flux
$
\FF(V) = (hu,hu^2+ %P 
gh^2/2+Gh\sigma_{zz}-Gh\sigma_{xx}
% ,\sigma_{xx}h^{2(1-\zeta)+1}u,\sigma_{zz}h^{2(\zeta-1)+1}u)
,hu\Xcal(\sigma_{xx}h^{2(1-\zeta)}),hu\Zcal(\sigma_{zz}h^{2(\zeta-1)})) 
\,.
$
Now, whereas
%weak 
{\it univoque} solutions to quasilinear (possibly non-conservative) systems
%should rather be
can be constructed %defined 
using (convex) entropies % following e.g. 
\cite{lefloch-2002},
conservative formulations alone (without admissibility criterion) % like entropy
are not enough. % except for linearly degenerate systems ? 
%
% That is why, to construct univoque BV solutions,we need to re-visit here
This is why %here 
we carefully investigate the building-block of univoque BV solutions: %we define 
univoque solutions to Riemann initial-value problems
for a quasilinear system \eqref{evolution}
% satisfied by $U=(h,hu,h\sigma_{xx},h\sigma_{zz})$
%formulated 
in well-chosen variable~$U$
\begin{equation}
\label{evolution}
\partial_t U + A(U) \partial_x U = S(U)
\end{equation}
in the homogeneous case $S\equiv0$ (obtained in the limit $\lambda\to\infty$).
Precisely, when $\zeta\le\frac12$ and $G>0$
we build %construct
%univoque {\it admissible} (i.e. entropy-dissipative)
the unique weak solutions $U(t,x)$ %\in\mathbb{R}^d$ ($d=4$)
{\it admissible under Lax condition}
to % the homogeneous (i.e. $S\equiv0$) 
Riemann problems for (\ref{evolution}) % , i.e. when it is supplied 
with piecewise-constant initial conditions 
\begin{equation}
\label{initial}
U(t\to0^+,x)=
\begin{cases}
U_l & x<0 
\\
U_r & x>0 
\end{cases}
\end{equation}
given \emph{any} states $U_l,U_r\in\Ucal$ in the strict hyperbolicity region $\Ucal=\{h>0,\sigma_{xx}>0,\sigma_{zz}>0\}\subset\R^d$.
%
% Interestingly,
Our Riemann 
%weak 
solutions %turn out to 
satisfy the conservative system (\ref{conservative})
% and the entropy dissipation \eqref{secondprinciple}
in the % usual weak (i.e. distributional)
distributional sense on $(t,x)\in\R^+\times\R$ % as usual
and are consistent with the standard Saint-Venant case when $G\to0$.
%
% \textcolor{red}{thanks to %because of 
% the particular %hyperbolic 
% structure of the system} ;
% ($F$ is actually convex along the one-dimensional % manifolds Hugoniot curves) ;
%
% Solutions satisfy our entropy dissipation rule by construction.
%
% Although weak solutions can then be constructed, it is well-known that BV weak solutions of quasilinear first-order systems are not unique,
% unless one adds some constraints (admissibility criteria) to be enforced during the construction process, 
% plus (see for instance \cite[14.10]{dafermos-2000}.
% an additional ``Tame-oscillation condition'' is required -- which cannot be ensured a priori,
% even in the subclass of BV solutions reached by the construction method, therefore also Lipschitz typically
%
%This implies that our admissibility criterion is enough to 
%This allows one to
These Riemann solutions are a key tool to
define % univoque
weak BV solutions to the Cauchy problem for % our quasilinear system
gSV % following the lines of e.g. \cite{berthon-lefloch-coquel}, 
% for small initial data % ????????????????????????????????????
which are %necessarily 
unique within the admissible % entropy-dissipative %-satisfying 
BV solutions'class modulo some %small 
restriction on oscillations, see e.g. \cite[Chap.X]{lefloch-2002}.
%
% \emph{smallness: for uniqueness of Riemann problems in the approximation process AND for compactness before passing to the limit ?}
% Gosse, Le Floch, Berthon, Coquel

% In addition, in view of the explicitly computed Riemann solutions, %we show that 
However, note that % as long as $G>0$ note that 
the vacuum state $h=0$ %, at the boundary of our open admissibility set $\Ucal$, 
shall never be reached 
as a limit state % at a finite space-time location 
by any sequence of admissible Riemann solutions when $G>0$,
as opposed to the standard Saint-Venant case $G=0$ 
% where a % bounded continuous
% sequence of admissible rarefaction waves can reach vacuum as a limit state in Riemann problem 
(like in the famous Ritter problem for instance).
%
% in the strictly non-Newtonian case $G>0$, %$\sigma_{xx}\neq\sigma_{zz}$,
% contrary to the standard shallow-water case $G=0$ % $\sigma_{xx}=\sigma_{zz}$ cannot be preserved
% where it is an admissible limit state through rarefaction waves,
% which is % an important observation
% a sever limitation for practical applications.
%
This is in fact related to well-posedness in the large (i.e. for any initial condition $U_l,U_r\in\Ucal$) when $G>0$,
as opposed to the standard Saint-Venant % shallow-water 
case $G=0$ (so the latter case is some kind of singular limit):
when $G>0$, gSV impulse blows up as $h\to0$, so vacuum cannot be reached, % in finite time
while Hugoniot curves in turn span the whole range $\Ucal$.
Also, consistently with the occurence of vacuum when $G=0$ (standard Saint-Venant),
the latter case can be recovered when $G\to0$ only for \emph{small} initial data
(otherwise, the intermediate state in Riemann solution blows up).
%
% $\sigma_{xx}=\sigma_{zz}$ cannot be preserved
%\textcolor{red}{We will address this issue for practical applications in a forthcoming work ?}
\section{Hyperbolic structure of the system of equations}
\sectionmark{Hyperbolic structure}
\label{sec:hyp}
Given $g>0,G\ge0$, consider first the gSV system (\ref{eq:gsv1}--\ref{eq:gsv2}--\ref{eq:gsv9}--\ref{eq:gsv10}) 
%The gSV quasilinear system $\partial_t U + A(U) \partial_x U = S(U)$ 
written in the non-conservative quasilinear form
(\ref{evolution}) % with $A$
using the variable $U=(h,u,\sigma_{xx},\sigma_{zz})\in\Ucal$.
One easily sees that $\lambda^0:=u$ is an eigenvalue with multiplicity two for the jacobian matrix $A$
associated with the linearly degenerate % two-dimensional 
$0$-characteristic field % spanned by
(i.e. $r^0\cdot\grad_U \lambda^0=0$) % for $i=1,2$
\begin{equation}
%\lambda^0 = u \quad
r^0 \in \spn\{r^0_1,r^0_2\} \quad
r^0_1 :=
\begin{pmatrix}
% G \\ 0 \\ (g + {N}/h) \\ 0
 Gh \\ 0 \\ (gh + N) \\ 0
\end{pmatrix}
\
r^0_2 :=
\begin{pmatrix}
 Gh \\ 0 \\ 0 \\ -(gh + N) % = -\partial_h P 
\end{pmatrix}
\end{equation}
with Riemann invariants $u,P$ (i.e. $r^0\cdot\grad_U P=0$).
% (obviously, $r^0\cdot\grad_U u=0$ since $\lambda^0=u$, and it also holds $r^0\cdot\grad_U P=0$).
% TO FIND RIEMANN INVARIANTS:
% - we have a system of two first-order PDEs in 3-space dimension (h,sx,sz)
% - the sx,sz derivatives are one another opposed, so it remains
% - one first-order PDE in 2-space dimension (h,sx) or (h,sz)
% - which can be solved by characteristics method for instance
Moreover, as long as $\zeta\le\frac12$, holds 
\begin{equation}
\label{dhp}
\partial_h P|_{\sigma_{xx}h^{2(1-\zeta)},\sigma_{zz}h^{2(\zeta-1)}} 
%= gh+G(1+2(1-\zeta))\sigma_{zz}-G(1+2(\zeta-1))\sigma_{xx}>0
= gh+G(\sigma_{zz}-\sigma_{xx})+2G(1-\zeta)(\sigma_{zz}+\sigma_{xx})>0
\end{equation}
for $h,\sigma_{xx},\sigma_{zz}>0$ so, after % standard  godwlesky-raviart
computations, the two other %remaining 
eigenvalues of $A$ are real
\begin{equation}
\lambda^\pm := u \pm \sqrt{ \partial_h P|_{\sigma_{xx}h^{2(1-\zeta)},\sigma_{zz}h^{2(\zeta-1)}} }
\end{equation}
and define two genuinely nonlinear fields (denoted by $+$ and $-$) spanned by
\begin{equation}
r^\pm :=
\begin{pmatrix}
 h \\ \pm \sqrt{ \partial_h P|_{\sigma_{xx}h^{2(1-\zeta)},\sigma_{zz}h^{2(\zeta-1)}} } \\ 2(\zeta-1)\sigma_{xx} \\ 2(1-\zeta)\sigma_{zz}
\end{pmatrix}
\end{equation}
with $\sigma_{xx}h^{2(1-\zeta)},\sigma_{zz}h^{2(\zeta-1)}$ % components of the conserative variables 
as Riemann invariants ; note in particular for $\zeta\in[0,\tfrac12]$ 
{%\small
\begin{equation}
\label{genuinelynonlinear}
r^\pm\cdot\grad_U \lambda^\pm = \pm
% \frac{ 3gh + 4(\zeta^2-7\zeta/2+3)G\sigma_{zz} + 2\zeta(1-2\zeta)G\sigma_{xx} }
\frac{ 3gh + 2G(3-2\zeta)(2-\zeta)\sigma_{zz} + 2G\zeta(1-2\zeta)\sigma_{xx} }
{ 2\sqrt{\partial_h P|_{\sigma_{xx}h^{2(1-\zeta)},\sigma_{zz}h^{2(\zeta-1)}}} } \gtrless 0
\,.
\end{equation}
}
Univoque % i.e. which are unique among piecewise smooth solutions
piecewise-smooth solutions 
to Cauchy-Riemann problems for (\ref{eq:gsv1}--\ref{eq:gsv2}--\ref{eq:gsv9}--\ref{eq:gsv10})
read $U(t,x)=\tilde U(x/t)$ % see e.g. \cite[section 4.2]{serre1-1996}
% in the STRICT hyperbolicity domain
% have the self-similar form 
with $\tilde U(\xi)$ piecewise differentiable solution on $\R\ni\xi$~to % the Ordinary Differential Equations (ODEs)
\begin{equation}
\label{eq:selfsimilar} 
\xi \tilde U' = A(\tilde U) \tilde U'
\qquad
\tilde U \underset{\xi\to-\infty}{\longrightarrow} U_l \,,\quad
\tilde U \underset{\xi\to+\infty}{\longrightarrow} U_r 
\end{equation}
% piecewise differentiable except on a finite number $M>0$ of discontinuities
having finitely-many %$M+1$ 
discontinuities $\xi_m$, $m=0\ldots M$,
shall next be constructed % explicitly  
%following the general procedure introduced by Lax \cite{lax-1957} 
for any % not necessarily small $|U_l-U_r|\ll1$ 
initial condition $U_l,U_r\in\Ucal$
using % centered 
elementary waves satisfying
$\tilde U' \in \spn{r^\pm}$, 
$\xi=\lambda^\pm$ % \emph{rarefaction waves}
% computed as integral curves after
therefore % using genuine nonlinearity
% \begin{equation}
% \label{rarefactionwave}
$
\tilde U'= %\frac
{r^\pm}/({r^\pm\cdot\grad_U \lambda^\pm}) 
$,
% \end{equation}
or
$\tilde U'=0$, % \emph{shocks}
and % $\xi \tilde U' = A(\tilde U) \tilde U'$
an % entropy 
admissibility criterion. % (free-energy dissipation)

\section{Elementary-waves solutions % to the system
}\sectionmark{Elementary waves} % use \sectionmark{} to alter or adjust the section heading in the running head
%Following \cite{lax-1957} we will construct piecewise smooth solutions to Riemann problems of the form
For all $U_l,U_r\in\Ucal$, unique % BV 
solutions to (\ref{evolution}--\ref{initial}) shall be constructed in the form
\begin{equation}
\label{solution}
\tilde U(\xi)=
\begin{cases}
U_l \equiv \tilde U_0 & \phantom{\xi_0<}\xi<\xi_0 
\\
\tilde U_1(\xi) & \xi_0<\xi<\xi_1 
\\
\cdots
\\
\tilde U_M(\xi) & \xi_{M-1}<\xi<\xi_M 
\\
U_r \equiv \tilde U_{M+1} & \xi_M<\xi\phantom{\xi_{M+1}}
\end{cases} 
\end{equation}
using $M$ differentiable states $\tilde U_m$ % $m=1\dots M$,
% with $M+1$ discontinuities and
to connect $U_l,U_r \in\Ucal$ through elementary waves.
% for the system is equipped with enough entropies (or kinetic relations), see \cite{lefloch-2002}, although $F$ is not a proper entropy for \eqref{conservative}

\subsection{Contact discontinuities and shocks}
Elementary-waves solutions (\ref{solution}) with a single discontinuity ($M=1$) shall be
%either 
$0$-contact discontinuities when,
denoting $\Upsilon_l$ (resp. $\Upsilon_r$) the left (resp. right) value of $\Upsilon$,
\begin{equation} 
\xi_0=u_l=u_r \quad P_l=P_r \quad \;
\end{equation}
or $\pm$-shocks 
when, 
denoting
% the extremal states 
% $U_l,U_r \in\Ucal$ are connected through a Hugoniot curve
$
% P_l = gh_l^2/2 + G Z^{-1}_{l} h_l^{1+2(1-\zeta)} - G X_{l} h_{l}^{1+2(\zeta-1)}
% \quad
% P_{r} = gh_{r}^2/2 + G Z^{-1}_{r} h_{r}^{1+2(1-\zeta)} - G X_{r} h_{r}^{1+2(\zeta-1)}
% P_l = gh_l^2/2 + G Z^{-1} h_l^{1+2(1-\zeta)} - G X h_{l}^{1+2(\zeta-1)}
% \quad
% P_{r} = gh_{r}^2/2 + G Z^{-1} h_{r}^{1+2(1-\zeta)} - G X h_{r}^{1+2(\zeta-1)}
% \text{ and } 
P_k = gh_k^2/2 + G Z^{-1} h_k^{1+2(1-\zeta)} - G X h_k^{1+2(\zeta-1)}
$, hold
\begin{eqnarray} 
% \begin{equation} 
% \label{rankinehugoniot}
% \begin{aligned}
\label{rankinehugoniot1}
& \xi_0 (h_{r}-h_l) = (h_{r}u_{r}-h_lu_l)\,,
\\
% & \partial_{t}(h u) + \partial_{x}\left( h u^2 + gh^2/2 + Gh(\sigla_{zz}-\sigla_{xx}) \right) = 0
% & \partial_{t}(h u) + \partial_{x}\left( h u^2 + gh^2/2 + hN \right) = 0
\label{rankinehugoniot2}
& \xi_0 (h_{r}u_{r}-h_lu_l) = (h_{r}u_{r}^2+P_{r}-h_lu_l^2-P_l)\,,
% \end{aligned}
% \end{equation}
\end{eqnarray}
with 2 constants
$Z^{-1}=\sigma_{zz,k}h_k^{2(\zeta-1)}>0, X=\sigma_{xx,k}h_k^{2(1-\zeta)}>0$ ($k\in\{l,r\}$),
%
% which are the values for two of the $\pm$-Riemann invariants.
% $Z^{-1}=\sigma_{zz,l}h_l^{2(\zeta-1)}=\sigma_{zz,r}h_{r}^{2(\zeta-1)}>0$; 
% $X=\sigma_{xx,l}h_l^{2(1-\zeta)}=\sigma_{xx,r}h_{r}^{2(1-\zeta)}>0$
%
thus
\begin{equation}
\label{hugoniotlocus}
u_{r} = u_l \pm \sqrt{(h_l^{-1}-h_{r}^{-1})(P_{r}-P_l)}
\end{equation}
on combining \eqref{rankinehugoniot1}, \eqref{rankinehugoniot2}.
Both waves satisfy %the 
Rankine-Hugoniot (RH) relationships
\begin{equation}
\label{rh}
\xi_0 ( V_r - V_l ) = \FF_r - \FF_l 
\end{equation}
and thus define standard weak solutions to \eqref{evolution} in the conservative variable $V(t,x)=\tilde V(x/t)$.
%
%  Rk: the standard ``weak'' definition of shocks is still possible with a jump in $X=\sigma_{xx}h^{2(1-\zeta)}$ and $Z=\sigma_{zz}h^{2(\zeta-1)}$ providing
%  $$ (h_{m+1}+h_m)(h_{m+1}u_{m+1}-h_mu_m) = (h_{m+1}-h_m)(h_{m+1}u_{m+1}+h_mu_m) $$
%  that is if and only if $h_{m+1}=0=h_m$, on noting $u_{m+1}=u_m=\xi_m$ implies the latter by \eqref{rankinehugoniot}.
%
Moreover, the entropy dissipation \eqref{secondprinciple} in the elastic limit $\lambda\to\infty$
\begin{equation}
\label{rhentropy}
E := - \xi_0 ( F_r - F_l ) +
%\left[ u(F + gh^2/2 + hN ) \right]^r_l 
u(F + gh^2/2 + hN )|_r - u(F + gh^2/2 + hN )|_l  \le 0
\end{equation}
can be checked for contact discontinuities (as an equality), 
%
% Note that contact discontinuities satisfy all companion conservation laws of gSV in the % standard weak
% distributional sense % see e.g. \cite[Lemma 1.7 p.10]{bouchut-2004}
% like the equality (\ref{secondprinciple}) in the limit $\lambda\to\infty$;
% they are therefore admissible by our % free-energy
% entropy-dissipation rule.
%
and for the
\emph{weak} shocks in the genuinely nonlinear fields $\lambda^\pm$ % with Riemann invariants $X,Z$ 
(i.e. shocks with small enough amplitude)
% which naturally extend the shocks from the standard (conservative) Saint-Venant system %recovered when $G=0$
which satisfy Lax admissibility condition, see \cite{lax-1957,dafermos-2000} and
% by standard computations 
\cite[(1.24) Chap. VI]{lefloch-2002}:
\begin{lemma}
Right and left states $V_r,V_l$ can be connected through an admissible
\begin{itemize}
 \item $-$-shock if $u_r = u_l - \sqrt{(h_l^{-1}-h_r^{-1})(P_r-P_l)}$, $h_r \ge h_l$
% $\lambda^-_l>\xi_0>\lambda^-_r$
%
% $(h_{m+1}-h_m,u_{m+1}-u_m)$ colinear\footnote{
%  A straightforward expansion in $o(h_{m+1}-h_m)$ of \eqref{hugoniotlocus} yields that it is the case  
%  when $u_{m+1} = u_m - \sqrt{(h_m^{-1}-h_{m+1}^{-1})(P_{m+1}-P_m)}$ if $h_{m+1} \ge h_m$
%  and $u_{m+1} = u_m + \sqrt{(h_m^{-1}-h_{m+1}^{-1})(P_{m+1}-P_m)}$ if $h_{m+1} \le h_m$.
% }
% to $r^-$ when $\tilde U_{m+1}\to\tilde U_m$ ; %which implies
% thus $h_{m+1} \ge h_m$ by genuine nonlinearity\footnote{
%  Since $r^-\cdot\grad_U\lambda^-<0$ the condition is satisfied only when $(h_{m+1}-h_m,u_{m+1}-u_m)$ is colinear to $r^-$
%  as $\tilde U_{m+1}\to\tilde U_m$, {\it and has same orientation}. 
% } \eqref{genuinelynonlinear}
 \item $+$-shock if $u_r = u_l - \sqrt{(h_l^{-1}-h_r^{-1})(P_r-P_l)}$, $h_r \le h_l$
% $\lambda^+_l>\xi_0>\lambda^+_r$
%
% $(h_{m+1}-h_m,u_{m+1}-u_m)$ colinear\footnote{
%  The expansion in $o(h_{m+1}-h_m)$ of \eqref{hugoniotlocus} yields that it is the case when
%  $u_{m+1} = u_m + \sqrt{(h_m^{-1}-h_{m+1}^{-1})(P_{m+1}-P_m)}$ if $h_{m+1} \ge h_m$
%  and $u_{m+1} = u_m - \sqrt{(h_m^{-1}-h_{m+1}^{-1})(P_{m+1}-P_m)}$ if $h_{m+1} \le h_m$.
% }
% to $r^+$ when $\tilde U_{m+1}\to\tilde U_m$ ;
% %which implies
% thus $h_{m+1} \le h_m$ by genuine nonlinearity\footnote{
%   Since $r^+\cdot\grad_U\lambda^+<0$ the condition is satisfied only when $(h_{m+1}-h_m,u_{m+1}-u_m)$ is colinear to $r^+$
%   as $\tilde U_{m+1}\to\tilde U_m$, {\it and has opposite orientation}.
% } \eqref{genuinelynonlinear}.
%
% Note that the branch $u_{m+1} - u_m = + \sqrt{(h_m^{-1}-h_{m+1}^{-1})(P_{m+1}-P_m)} \ge 0$ of \eqref{hugoniotlocus}
% is never admissible, while $u_{m+1} - u_m = - \sqrt{(h_m^{-1}-h_{m+1}^{-1})(P_{m+1}-P_m)} \le 0$ is always admissible
% and corresponds to a $+$-schock wave (with tangent $r^+$) when $h_{m+1} \le h_m$ 
% and to a $-$-schock wave (with tangent $r^-$) when $h_{m+1} \ge h_m$.
\end{itemize}
\end{lemma}
Indeed, it is enough that $F|_{X,Z}$ is convex in $(h,hu)$ to discriminate against non-physical (weak) shocks,
% Now, $F|_{X,Z}$ is convex in $(h,hu)$ if and only if % insofar as
or equivalently, that $F|_{X,Z}/h$ is convex in $(h^{-1},u)$ \cite{bouchut-2004}.
\begin{proof} % On recalling \eqref{freeenergy}
$
2 F/h %\frac{F}h %= u^2 + gh +  G(\sigma_{xx}+\sigma_{zz}-\ln\sigma_{xx}-\ln\sigma_{zz}-2)
= u^2 + gh +  G(\sigma_{xx}+\sigma_{zz}-\ln\sigma_{xx}-\ln\sigma_{zz}-2)
$
is convex in $(h^{-1},u)$ if %and only if the following is positive
$$
%\partial^2_{h^{-2}} \frac{F|_{X,Z}}h = \frac2{h^{-3}} \partial_h \frac{F|_{X,Z}}h + \frac1{h^{-4}} \partial^2_{h^2} \frac{F|_{X,Z}}h
\partial^2_{h^{-2}} \frac{F|_{X,Z}}h = \frac2{h^{-3}} \partial_h \frac{F|_{X,Z}}h + \frac1{h^{-4}} \partial^2_{h^2} \frac{F|_{X,Z}}h
$$
% now it holds for $\zeta \in [0,1/2]$
is positive, which holds when $\zeta \in [0,1/2]$ such that
$$
2\partial_h \frac{F|_{X,Z}}h = g + G ( 2(1-\zeta) Z h^{2(1-\zeta)-1} + 2(\zeta-1) X  h^{2(\zeta-1)-1} ) \ge 0 \,,
$$
$$
2\partial^2_{h^2} \frac{F|_{X,Z}}h = G ( 2(1-\zeta)(2(1-\zeta)-1) Z h^{-2\zeta)} + 2(\zeta-1)(2(\zeta-1)-1) X  h^{2(\zeta-2)} ) \ge 0 \,.
$$ 
\end{proof}

\subsection{Rarefaction waves}
Elementary waves with two discontinuities ($M=2$) which are not a combination of two elementary waves with one discontinuity each shall be,
on noting $k\in\{l,r\}$,
\begin{itemize}
 \item a $+$-rarefaction wave if $h_l=h_0<h_r=h_2$ such that for all $\xi\in(\xi_0\equiv\lambda^+_l,\xi_2\equiv\lambda^+_r)$
{\small
\begin{eqnarray*}
\xi 
& = & \lambda^+_k + \int_{h_k}^{h_1(\xi)} dh
\frac
{ 3gh + (4\zeta^2-14\zeta+12)GZ^{-1} h^{2(1-\zeta)} + 2\zeta(1-2\zeta)GX h^{2(\zeta-1)} }
{ 2 h \sqrt{gh+(1+2(1-\zeta))GZ^{-1} h^{2(1-\zeta)}-(1+2(\zeta-1))GX h^{2(\zeta-1)}} }
\\
u_1(\xi)
& = & u_k + \int_{h_k}^{h_1(\xi)} dh \sqrt{ gh^{-1}+(1+2(1-\zeta))GZ^{-1} + h^{-2\zeta}-(1+2(\zeta-1))GX h^{2(\zeta-2)} } \,,
\end{eqnarray*}
}
% \begin{eqnarray*}
% \xi 
% & = \lambda^+_l + \int_{h_l}^{h_1(\xi)} dh
% \frac
% { 3gh + (4\zeta^2-14\zeta+12)G/Z%_l
%   h^{2(1-\zeta)} + 2\zeta(1-2\zeta)GX%_l
%   h^{2(\zeta-1)} }
% { 2 h \sqrt{gh+(1+2(1-\zeta))G/Z%_l
%   h^{2(1-\zeta)}-(1+2(\zeta-1))GX%_l
%   h^{2(\zeta-1)}} }
% \\
% & = \lambda^+_r + \int_{h_r}^{h_1(\xi)} dh
% \frac
% { 3gh + (4\zeta^2-14\zeta+12)G/Z%_r
%   h^{2(1-\zeta)} + 2\zeta(1-2\zeta)GX%_r
%   h^{2(\zeta-1)} }
% { 2 h \sqrt{gh+(1+2(1-\zeta))G/Z%_r
%   h^{2(1-\zeta)}-(1+2(\zeta-1))GX%_r
%   h^{2(\zeta-1)}} }
% \end{eqnarray*}
% This integral curve corresponds to the third Riemann invariant of the $\pm$-fields, which is unfortunately not analytically integrable.
% \begin{eqnarray*}
% u_1(\xi)
% & = u_l + \int_{h_l}^{h_1(\xi)} dh \sqrt{ 
%  gh^{-1}+(1+2(1-\zeta))G/Z%_0
%  h^{-2\zeta}-(1+2(\zeta-1))GX%_0
%  h^{2(\zeta-2)} }
% \\
% & = u_r + \int_{h_r}^{h_1(\xi)} dh \sqrt{ 
%  gh^{-1}+(1+2(1-\zeta))G/Z%_0
%  h^{-2\zeta}-(1+2(\zeta-1))GX%_0
%  h^{2(\zeta-2)} } 
% \end{eqnarray*}
 \item a $-$-rarefaction wave if $h_l=h_0>h_r=h_2$ such that for all $\xi\in(\xi_0\equiv\lambda^-_l,\xi_2\equiv\lambda^-_r)$
{\small
\begin{eqnarray*}
\xi & = & \lambda^-_k - \int_{h_k}^{h_1(\xi)} dh \frac
{ 3gh + (4\zeta^2-14\zeta+12)GZ^{-1} h^{2(1-\zeta)} + 2\zeta(1-2\zeta)GX h^{2(\zeta-1)} }
{ 2 h \sqrt{gh+(1+2(1-\zeta))GZ^{-1} h^{2(1-\zeta)}-(1+2(\zeta-1))GX h^{2(\zeta-1)}} }
\\
u_1(\xi) & = & u_k - \int_{h_k}^{h_1(\xi)} dh \sqrt{ gh^{-1}+(1+2(1-\zeta))GZ^{-1} h^{-2\zeta}-(1+2(\zeta-1))GX h^{2(\zeta-2)} } \,.
\end{eqnarray*}
}
%  \begin{eqnarray*}
% \xi 
% & = \lambda^-_l - \int_{h_l}^{h_1(\xi)} dh
% \frac
% { 3gh + (4\zeta^2-14\zeta+12)G/Z%_l
%   h^{2(1-\zeta)} + 2\zeta(1-2\zeta)GX%_l
%   h^{2(\zeta-1)} }
% { 2 h \sqrt{gh+(1+2(1-\zeta))G/Z%_l
%   h^{2(1-\zeta)}-(1+2(\zeta-1))GX%_l
%   h^{2(\zeta-1)}} }
% \\
% & = \lambda^-_r - \int_{h_r}^{h_1(\xi)} dh
% \frac
% { 3gh + (4\zeta^2-14\zeta+12)G/Z%_r
%   h^{2(1-\zeta)} + 2\zeta(1-2\zeta)GX%_r
%   h^{2(\zeta-1)} }
% { 2 h \sqrt{gh+(1+2(1-\zeta))G/Z%_r
%   h^{2(1-\zeta)}-(1+2(\zeta-1))GX%_r
%   h^{2(\zeta-1)}} }
% \end{eqnarray*}
% This integral curve corresponds to the third Riemann invariant of the $\pm$-fields, which is unfortunately not analytically integrable.
% \begin{eqnarray*}
% u_1(\xi)
% & = u_l - \int_{h_l}^{h_1(\xi)} dh \sqrt{ 
%  gh^{-1}+(1+2(1-\zeta))G/Z%_0
%  h^{-2\zeta}-(1+2(\zeta-1))GX%_0
%  h^{2(\zeta-2)} }
% \\
% & = u_r - \int_{h_r}^{h_1(\xi)} dh \sqrt{ 
%  gh^{-1}+(1+2(1-\zeta))G/Z%_0
%  h^{-2\zeta}-(1+2(\zeta-1))GX%_0
%  h^{2(\zeta-2)} } 
% \end{eqnarray*}
\end{itemize}
%given two constants $Z^{-1}=\sigma_{zz}h^{2(\zeta-1)}|_k>0$ and $X=\sigma_{xx}h^{2(1-\zeta)}|_k>0$, $k\in\{l,r\}$.

\section{Solution to the general Riemann problem}
\sectionmark{Solution}
The general Riemann problem can be solved by combining elementary waves \cite{lax-1957}.
% Recalling admissible (weak) shocks are characterized by Lax conditions.
Solutions \eqref{solution} % with finitely many discontinuities 
to systems with 3 characteristic fieds require 3 backward characteristics %well-defined 
through all points in $t>0$,
except %maybe 
on discontinuities, that are:
$\xi_0\le\xi_1$ %therefore 
associated with %a discontinuity in 
the $-$-field,
$\xi_2 \in (\xi_1,\xi_3)$ associated with %a discontinuity in 
the $0$-field, and
$\xi_3\le\xi_4$ associated with %a discontinuity in 
the $+$-field.
% (possibly with % degeneracies 
% $\xi_0=\xi_1$, $\xi_3=\xi_4$).
%
So finally, %in view of our elementary waves, for all initial conditions, there exist a unique 
a solution to the Riemann problem is characterized by 
\begin{equation}
\label{characterization}
X_l=X_1=X_2, \ Z_l=Z_1=Z_2 \quad u_2 = u_3, \ P_2 = P_3 \quad X_r=X_4=X_3, \ Z_r=Z_4=Z_3 
\end{equation}
% $u,P$ are %Riemann 
% invariants for the linearly degenerate field $\lambda^0$. 
with a $(h_2,u_2)$-locus given by
% $$
% u_2 = u_l - \sqrt{(h_l^{-1}-h_2^{-1})(P_2-P_l)} \quad h_2 \ge h_l %  \text{ when } h_2 \ge h_l
% $$
{\small
$
u_2 = u_l - \sqrt{(h_l^{-1}-h_2^{-1})(P_2-P_l)}
$}
when $h_2 \ge h_l$
and
{\small
$$
u_2 \leftarrow u_1(\xi) = u_l - \int_{h_l}^{h_1(\xi)} dh \sqrt{ gh^{-1}+(1+2(1-\zeta))GZ_l^{-1}h^{-2\zeta}-(1+2(\zeta-1))GX_lh^{2(\zeta-2)} }
$$
$$
\xi \ge \lambda_l - \int_{h_l}^{h_1(\xi)} dh
\frac
{ 3gh + (4\zeta^2-14\zeta+12)GZ_l^{-1}h^{2(1-\zeta)} + 2\zeta(1-2\zeta)GX_lh^{2(\zeta-1)} }
{ 2 h \sqrt{gh+(1+2(1-\zeta))GZ_l^{-1}h^{2(1-\zeta)}-(1+2(\zeta-1))GX_lh^{2(\zeta-1)}} }
% \quad h_2 \leftarrow h_1(\xi)\le h_l
$$
}\noindent
when $h_2 \leftarrow h_1(\xi)\le h_l$; a $(h_3,u_3)$-locus given by
% $$
% u_3 = u_r + \sqrt{(h_r^{-1}-h_3^{-1})(P_3-P_r)} \quad h_3 \ge h_r
% $$
{\small
$
u_3 = u_r + \sqrt{(h_r^{-1}-h_3^{-1})(P_3-P_r)}
$}
on the other hand when $h_3 \ge h_r$ and, when $h_3 \leftarrow h_4(\xi)\le h_r$,
{\small
$$
u_3 \leftarrow u_4(\xi) = u_r + \int_{h_r}^{h_4(\xi)} dh \sqrt{ gh^{-1}+(1+2(1-\zeta))GZ_r^{-1}h^{-2\zeta}-(1+2(\zeta-1))GX_rh^{2(\zeta-2)} }
$$
$$
\xi \le \lambda_r + \int_{h_r}^{h_4(\xi)} dh
\frac
{ 3gh + (4\zeta^2-14\zeta+12)GZ_r^{-1}h^{2(1-\zeta)} + 2\zeta(1-2\zeta)GX_rh^{2(\zeta-1)} }
{ 2 h \sqrt{gh+(1+2(1-\zeta))GZ_r^{-1}h^{2(1-\zeta)}-(1+2(\zeta-1))GX_rh^{2(\zeta-1)}} } \,.
% \quad h_3 \leftarrow h_4(\xi)\le h_r
$$
}
\begin{theorem}
Given $\xi\in\left[0,\frac12\right]$, $g>0$, $G>0$,
the Riemann problem for gSV admits a unique admissible weak solution in $\Ucal$ for all $U_l,U_r\in\Ucal$\;;
this solution is piecewise continuous and differentiable with at most $5$ discontinuity lines in $(x,t)\in\R\times\R^+$.
\end{theorem}
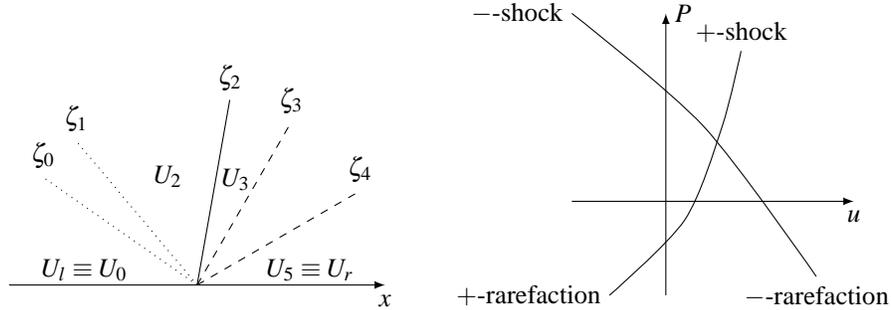
\begin{tikzpicture}[scale=0.5]
\draw[-latex](-5,0) -- (5,0) node[below] {$x$};
\node at (-3,0.4) {$U_l \equiv U_0$};
\node at (3,0.4) {$U_5 \equiv U_r$};
\draw[dashed](0,0)--(30:5) node[above] {$\zeta_4$};
\node at(70:3){$U_3$};
\draw[dashed](0,0)--(60:5) node[above] {$\zeta_3$};
\draw(0,0)--(80:5) node[above] {$\zeta_2$};
\node at(105:3){$U_2$};
\draw[dotted](0,0)--(130:5) node[above] {$\zeta_1$};
\draw[dotted](0,0)--(145:5) node[above] {$\zeta_0$};
\end{tikzpicture}
\hspace{.5cm}
\begin{tikzpicture}[scale=0.25]
\draw[-latex](-5,0) -- (10,0) node[below]{$u$};
\draw[-latex](0,-5) -- (0,10) node[right]{$P$};
\draw plot [smooth] coordinates {(-5,10) (2,4) (8,-4)};
\draw plot [smooth] coordinates {(-3,-5) (1,-1) (3,4) (4,8)};
\node[above] at (4,8){$+$-shock};
\node[left] at (-3,-5){$+$-rarefaction};
\node[left] at (-5,10){$-$-shock};
\node[below] at (8,-4){$-$-rarefaction};
\end{tikzpicture}
%
% \begin{figure}[b]
% \sidecaption
% \includegraphics[scale=.65]{figure}
% % If no graphics program available, insert a blank space i.e. use \picplace{5cm}{2cm} % Give the correct figure height and width in cm
% \caption{If the width of the figure is less than 7.8 cm use the \texttt{sidecapion} command to flush the caption on the left side of the page. If the figure is positioned at the top of the page, align the sidecaption with the top of the figure -- to achieve this you simply need to use the optional argument \texttt{[t]} with the \texttt{sidecaption} command}
% \label{fig:1}       % Give a unique label
% \end{figure}
\begin{proof}
It suffices to show that there exists one unique solution satisfying \eqref{characterization} for all $U_l,U_r\in\Ucal$.
%such that $u_2 = u_3$ and $P_2 = P_3$ with the characterization above.
%
Now, it holds $\partial_h P>0$ % by assumption 
and 
% on noting 
% $P|_{X,Z}\to-\infty$ as $h\to0^+$, $P|_{X,Z}\to+\infty$ as $h\to+\infty$ when $\xi\in\left[0,\frac12\right)$,
% $P|_{X,Z}\to0^+$ as $h\to0^+$, $P|_{X,Z}\to+\infty$ as $h\to+\infty$ when $\xi=\frac12$,
% to visualize this existence result 
one can use % the variable 
$(u,P,X%:=\sigma_{xx}h^{2(1-\zeta)}
,Z)%:=(\sigma_{zz}h^{2(\zeta-1)})^{-1})
\in\R\times\R\times\R^+_\star\times\R^+_\star% \cup \{0,0,0\}
$ %($Z=0=X$ is the single admissible value when $P=0$), %=h
as parametrization of the state space $\Ucal$
(with $P\in\R^+_\star$ when $\xi=\frac12$), see figures above.
Moreover, $\partial_P u=(\partial_h P)^{-1}\partial_h u$ 
% $\partial_h u = - \left( \partial_h P (h_k^{-1}-h^{-1}) + h^{-2}(P_k-P) \right)/(2\sqrt{})$
is negative along the $(h_2,u_2)$-locus, strictly except at $(h_l,u_l)$,
and positive along the $(h_3,u_3)$-locus, strictly except at $(h_r,u_r)$.
This is indeed easily established using $\partial_h u=(\partial_\zeta h)^{-1}\partial_\zeta u$ for rarefaction part; % of the locus
$\partial_h u=\pm\frac{P_*-P+h^2\partial_h P(h^{-1}-h^{-1}_*)}{2\sqrt{(h_*^{-1}-h^{-1})(P-P_*)}}$
for shock part, where $\partial_h P>0$ and 
$P$ is monotone increasing while $h^{-1}$ is monotone decreasing % (thus $P\lesseqgtr P_*,h^{-1}\gtreqless h^{-1}_*$). 
thus $P\ge P_*,h^{-1}\le h^{-1}_*$ when $h\ge h_*$ with $*=l/r$.

So finally, 
since $(u_3|_{X_r,Z_r}-u_2|_{X_l,Z_l})\to-\infty$ as $h=h_2=h_3\to0^+$ % $P=P_2=P_3\to-\infty$ when $\xi\in\left[0,\frac12\right)$, or ...
and 
$(u_3|_{X_r,Z_r}-u_2|_{X_l,Z_l})\to+\infty$ as $h=h_2=h_3\to+\infty$, %$P=P_2=P_3\to+\infty$,
there exists one, and only one, % a unique 
$P=P_2=P_3$ zero of the continuous %monotonically
strictly non-decreasing function $(u_3|_{X_r,Z_r}-u_2|_{X_l,Z_l})$.
\end{proof}

Note %however 
that % as usual for isentropic Euler equations
it is not clear yet whether the unique Riemann solutions constructed above under Lax admissibility condition 
always satisfy the entropy dissipation \eqref{secondprinciple}.
Classically, this % can be
is ensured for weak shocks only, % see dafermos-2000 for instance
using the asymptotic expansion of the \emph{convex} entropy $F$ as usual (see e.g. \cite[Chap. VI]{lefloch-2002})
like in Saint-Venant case $G=0$ with small initial data.
% ONE COULD COMPUTE THE VARIATION OF ENTROPY RATE ALONG HUGONIOT SHOCK CURVES PARAMETRIZED BY H>H_*... LENGTHY BUT SHOULD WORK (1d) !
%Anyway, and 
Interestingly, the latter limit case can be recovered in the limit $G\to0^+$ also for small initial data only.

\begin{corollary}
When $G\to0^+$ one recovers the usual Riemann solution to the standard Saint-Venant system $G=0$
($\sigma_{xx},\sigma_{zz}$ then being % so-called
``passive tracers'') only for initial data such that $U_l,U_r$ are close enough within $\Ucal$.
In particular, it is not possible to reach % univoque % ??
piecewise continuous and differentiable Riemann solutions with a vacuum state $h=0$ 
as the limits of bounded continuous sequences of Riemann solutions when $G>0$, 
as opposed to the standard Saint-Venant case $G=0$.
%
% In view of the explicit formulation of the Riemann solution,
% it is not possible to extend its definition to vacuum
% (when either of the two states $U_l,U_r$ goes a limit state with $h=0$ at the boundary of $\Ucal$).
%
\end{corollary}
\begin{proof}
The limit $h\to0$ can only be reached through rarefaction waves.
When $G=0$, this necessarily occurs for large initial data.
But when $G>0$, % $\xi\in\left[0,\frac12\right]$ 
the integrals defining the rarefaction waves are not well-defined (bounded)
as $h_1\to0$ ($-$-field) or $h_4\to0$ ($+$-field),
so this cannot occur for bounded (continuous sequences of) solutions.
\end{proof}
% We will come back to this issue %point
% in a forthcoming work \cite{boyaval-2016}.
% % We address this issue in the work in preparation \cite{boyaval-2016}.
% % where we will study another gSV system to solve the vacuum-admissibility problem:
% % we will use a physically-natural Finite-Extesible Nonlinear Elastic (FENE) model.
% \section{Conclusion}
% When $\xi\in\left[0,\frac12\right]$, the gSV-JS model is hyperbolic.
% %
% Then, when $G>0$ and for all initial conditions \eqref{initial} in % the strict hyperbolicity region
% $\Ucal$, 
% we could construct a unique weak solution % in the standard distributional sense for the conservative system
% satisfying Lax admissibility % criterion
% condition,
% which is equivalent to the dissipation rule \eqref{secondprinciple} for the so-called weak shocks. % $|U_l-U_r|$ sufficiently small
% %
% They coincide with the usual solutions of the standard Saint-Venant case (recovered in the limit $G\to0$).
% %
% When $G>0$, similarly to the standard Saint-Venant case $G=0$,
% we do not know whether entropy-dissipation actually hold for all Lax-admissible solutions.
% %
% But unlike the standard Saint-Venant case $G=0$, 
% when $G>0$, can occur.

\bibliographystyle{amsplain}

\providecommand{\bysame}{\leavevmode\hbox to3em{\hrulefill}\thinspace}
\providecommand{\MR}{\relax\ifhmode\unskip\space\fi MR }
\providecommand{\MRhref}[2]{%
  \href{http://www.ams.org/mathscinet-getitem?mr=#1}{#2}
}
\providecommand{\href}[2]{#2}

\end{document}